\documentclass[11pt]{amsart}
\usepackage{amsfonts,amsmath,amsthm,mathrsfs}
\newtheorem{thm}{Theorem}[section]
 \newtheorem{cor}[thm]{Corollary}
 \newtheorem{lem}[thm]{Lemma}{\rm}

 {\rm}

\numberwithin{equation}{section}

\def\x{\mathbf{x}}
\def\f{\mathbf{f}}

\def\y{\mathbf{y}}

\def\v{\mathbf{v}}
\def\K{\mathbf{K}}
\def\M{\mathbf{M}}
\def\B{\mathbf{B}}

\def\R{\mathbb{R}}
\def\N{\mathbb{N}}
\def\D{\mathbf{D}}

\def\bxi{\boldsymbol{\xi}}
\def\bsig{\boldsymbol{\sigma}}

\begin{document}
\title{The Moment-SOS hierarchy and the Christoffel-Darboux kernel}
\thanks{Work partly funded by the AI Interdisciplinary Institute ANITI through the French ``Investing for the Future PI3A" program under the Grant agreement ANR-19-PI3A-0004}
\author{Jean B. Lasserre}
\address{LAAS-CNRS and Institute of Mathematics\\
University of Toulouse\\
LAAS, 7 avenue du Colonel Roche\\
31077 Toulouse C\'edex 4, France\\
email: lasserre@laas.fr}

\maketitle

\begin{abstract}
We consider the global minimization of a polynomial on a compact set $\B$.
We  show that each step of the Moment-SOS hierarchy has a nice and simple interpretation that complements the usual one. 
Namely,
it computes coefficients of a polynomial in an orthonormal basis of $L^2(\B,\mu)$ where $\mu$ is an arbitrary  reference measure whose support is exactly $\B$. 
The resulting polynomial is a certain density (with respect to 
$\mu$) of some signed measure on $\B$. 
When some relaxation is exact (which generically takes place)
the coefficients of the optimal polynomial density are values of orthonormal polynomials at the global minimizer
and the optimal (signed) density is simply related to the Christoffel-Darboux (CD) kernel and the Christoffel function associated with $\mu$.
In contrast to the hierarchy of upper bounds which computes \emph{positive} densities, the global optimum can be achieved exactly as integration against a polynomial (signed) density because the CD-kernel is a reproducing kernel,
and so can mimic a Dirac measure (as long as finitely many moments are concerned).
\end{abstract}

\section{Introduction}

Consider the \emph{Polynomial Optimization Problem} (POP):
\[f^*\,=\,\min_\x \{\,f(\x):\: \x\in \B\,\},\]
where $\B\subset\R^n$ is a compact basic semi-algebraic set.
For the hierarchy of upper bounds discussed below, $\B$ is restricted to be a ``simple" set like e.g. a box, an ellipsoid, a simplex, a discrete-hypercube, or their
image by an affine transformation. Indeed,
to define an SOS-hierarchy of upper bounds converging to the global minimum $f^*$ 
as described in e.g. \cite{etienne-monique-2020,new-look,slot-monique-2020}, we use a measure $\mu$ whose support is exactly $\B$, and
for which all moments 
\[\mu_\alpha:=\int_\B \x^\alpha\,d\mu\,,\quad\alpha\in\N^n\,,\]
can be obtained numerically or in closed-form. For instance if $\B$ is a box, an ellipsoid or a simplex,
$\mu$ can chosen to be the Lebesgue measure restricted to $\B$. On the hypercube
$\{-1,1\}^n$ $\mu$ one may choose for $\mu$ the counting measure, etc.

\subsection{Background} 
Let $\B\subset\R^n$ be the basic semi-algebraic set defined by
\begin{equation}
\label{set-B}
\B\,=\{\,\x\in\R^n\::\: g_j(\x)\geq0\,\,\quad j=1,\ldots,m\,\}\,,
\end{equation}
for some polynomials $g_j\in\R[\x]$, $j=1,\ldots,m$. Let $g_0(\x)=1$ for all $\x$, and let
$d_j:=\lceil{\rm deg}(g_j)/2\rceil$, $j=0,\ldots,m$. Define $\Sigma[\x]_t$ be the set of sums-of-squares (SOS) polynomials of degree at most $2t$.

\subsection*{A hierarchy of lower bounds}
To approximate $f^*$ from below, consider the hierarchy of semidefinite programs indexed by $t\in\N$:
\begin{equation}
\label{eq:lower}
\rho_t=\sup_{\lambda\,\sigma_j}\,\{\,\lambda\::\: f-\lambda=\sum_{j=0}^m\psi_j\,g_j\,;\quad \psi_j\in\Sigma[\x]_{t-d_j}\,,\:j=0,\ldots,m\,\}\,,
\end{equation}
where $\Sigma[\x]_t$ denotes the space of sum-of-squares (SOS) polynomials of degree at most $2t$.
Under some Archimedean assumption on the $g_j$'s, 
$\rho_t\leq f^*$ for all $t$ and the sequence 
of lower bounds $(\rho_t)_{t\in\N}$ is monotone non decreasing and converges to $f^*$ as $t$ increases. Moreover, 
by a result of Nie \cite{nie}, its convergence is finite generically, and global minimizers can be extracted from an optimal solution of the 
semidefinite program which is the dual of \eqref{eq:lower}; see e.g. \cite{lass-book}. The sequence of semidefinite programs \eqref{eq:lower} and their duals, both indexed by $t$,
forms what is called the Moment-SOS hierarchy initiated in the early 2000's. For more details 
on the Moment-SOS hierarchy and its numerous applications in and outside optimization, the interested reader is referred to \cite{lass-siopt-01,lass-book}.

\subsection*{A hierarchy of upper bounds}

Let $\mu$ be a finite Borel measure whose support is exactly $\B$, where now 
$\B$ is a ``simple" set as mentioned earlier. (Hence all moments of $ \mu$ are available in closed form.)
To approximate $f^*$ from above, consider the hierarchy of semidefinite programs
\begin{equation}
\label{eq:upper}
u_t=\inf_{\sigma}\,\{\,\displaystyle \int_\B f\,\sigma\,d\mu\::\: \int_\B \sigma\,d\mu\,=\,1\,;\: \sigma\in\Sigma[\x]_t\,\}\,.
\end{equation}

That $u_t\geq f^*$ is straightforward since 
\[f\geq f^*\mbox{ on $\B$ }\Rightarrow \int_\B f\,\sigma\,d\mu\,\geq\,f^*\,\int_\B \sigma\,d\mu\,=\,f^*\,,\]
for any feasible SOS $\sigma$.
In \cite{new-look} it was proved that $u_t\downarrow f^*$ as $t$ increases, and in fact solving 
the dual of \eqref{eq:upper} is solving a generalized eigenvalue problem for a certain pair of real symmetric matrices. 
In a series of papers, de Klerk, Laurent an co-workers have provided several
rates of convergence of $u_t\downarrow f^*$ for several examples of sets $\B$.
For more details and results, the interested reader is referred to \cite{etienne-monique-2020,slot-monique-2020,slot-monique-2020-2,slot-monique-2020-3} and references therein.\\

The meaning of \eqref{eq:upper} is clear if one recalls that
\begin{equation}
\label{eq:f-star-measure}
f^*\,=\,\inf_\phi\,\{\,\int_\B f\,d\phi\::\: \phi(\B)=1\,;\:\phi\in\mathscr{M}(\B)_+\,\},
\end{equation}
where $\mathscr{M}(\B)_+$ is the space of \emph{all} finite Borel measures on $\B$.
Indeed in \eqref{eq:upper} one only considers the (restricted) subset of probability measures on $\B$ that have
a density (an SOS of degree at most $2t$) with respect to $\mu$ whereas in \eqref{eq:f-star-measure} one considers \emph{all}
probability measures on $\B$. In particular, the Dirac measure $\phi:=\delta_{\bxi}$ at any global
minimiser $\bxi\in\B$ belongs to $\mathscr{M}(\B)_+$ but does \emph{not} have a density with respect to $\mu$, 
which explains why the convergence $u_t\downarrow f^*$ as $t$ increases, can be only asymptotic and not finite; an exception
is  when $\B$ is a finite set (e.g. $\B=\{-1,1\}^n$ and $\mu$ is the counting measure).

\subsection{Contribution}

Our contribution is to show that in fact the dual of the semidefinite program
\eqref{eq:lower} for computing the lower bound $\rho_t$ has also an
interpretation of the same flavor as \eqref{eq:upper} where one now considers
\emph{signed} Borel measures $\phi_t$ with a distinguished \emph{polynomial} density with respect to $\mu$. Namely, 
the dual of \eqref{eq:lower} minimizes $\int_\B fd\phi_t$ over signed measures $\phi_t$ of the form:
\begin{equation}
\label{eq:form}
d\phi_t(\x)\,=\,\sigma_t(\x)\,d\mu(\x)\,=\,\left(\sum_{\vert\alpha\vert\leq 2t}\sigma_\alpha\,T_\alpha(\x)\right)\,d\mu(\x)\,,
\end{equation}
where :

-  $(T_\alpha)\subset\R[\x]$ is a family of polynomials that are orthonormal with respect to $\mu$, and

- the coefficients $\bsig_t=(\sigma_\alpha)_{\alpha\in\N^n_{2t}}$ of the polynomial $\sigma_t\in\R[\x]_{2t}$ satisfy the usual semidefinite constraints that are necessary for $\bsig_t$ to be 
\emph{moments} of a measure on $\B$.

Eventually for some $t\in\N$, $\bsig_t$  satisfies:
\begin{equation}
\label{ideally}
\sigma_\alpha\,=\,T_\alpha(\bxi)\,=\,\int_\B T_\alpha(\x)\,\delta_{\bxi}(d\x)\,,\quad \vert\alpha\vert\,\leq\,2t\,,
\end{equation}
where $\bxi$ is an arbitrary global minimizer and $\delta_{\bxi}$ is the Dirac measure at $\bxi\in\B$. Indeed then
\[\int_\B f(\x)\,d\phi_t(\x)\,:=\,\int_\B f(\x)\,\sum_{\vert\alpha\vert\leq 2t}T_\alpha(\bxi)\,T_\alpha(\x)\,d\mu(\x)\,=\,f(\bxi)\,,\]
because the Christoffel-Darboux Kernel $K_t(\x,\y):=\sum_{\vert\alpha\vert\leq 2t}T_\alpha(\x)\,T_\alpha(\y)$ is a \emph{reproducing} kernel 
for $\R[\x]_{2t}$, considered to be a finite-dimensional subspace of the Hilbert space $L^2(\B,\mu)$. Moreover,
$\sigma_t(\bxi)^{-1}$ is nothing less than the Christoffel function evaluated at the global minimizer $\bxi$ of $f$ on $\B$.

\emph{As a take home message and contribution of this paper,
it turns out that the dual
of the step-$t$ semidefinite relaxation \eqref{eq:lower} is a semidefinite program that computes the coefficients $\bsig_t=(\sigma_\alpha)$
of the polynomial density $\sigma_t$  in \eqref{eq:form}. In addition, when the relaxation is exact then $\sigma_t(\bxi)^{-1}$  is the Christoffel function of $\mu$, evaluated at a global minimizer $\bxi$ of $f$ on $\B$.}

Interestingly, in the dual of \eqref{eq:lower} there is no mention of the reference measure $\mu$. 
Only \emph{after} we fix some arbitrary reference measure $\mu$ on $\B$, we can interpret an optimal solution
as coefficients $\bsig_t$ of an appropriate polynomial density with respect to $\mu$.

So in both \eqref{eq:upper} and the dual of \eqref{eq:lower}, one searches for a polynomial ``density" with respect to $\mu$. 
In \eqref{eq:upper} one  searches for 
an SOS density (hence a  positive density) whereas in the dual of \eqref{eq:lower} one searches for a \emph{signed} polynomial density 
whose coefficients (in the basis of orthonormal polynomials)  are moments of a measure on $\B$ (ideally 
the Dirac at a global minimizer). 

The advantage of the (signed) polynomial density in \eqref{eq:form} compared to the 
(positive) SOS density in \eqref{eq:upper}, is to be able to obtain the global optimum
$f^*$ as the integral of $f$ against this density, which is impossible with the SOS density of \eqref{eq:upper}.

At last but not least, this interpretation  establishes another (and rather surprising) simple link between 
polynomial optimization (here the Moment-SOS hierarchy), the Christoffel-Darboux kernel and  the Christoffel function, 
fundamental tools in the theory of orthogonal polynomials and the theory of approximation. 
Previous contributions in this vein  include
\cite{lass-2020} to characterize upper bounds \eqref{eq:upper}, \cite{etienne-monique-2020,slot-monique-2020,slot-monique-2020-2} 
to analyze their rate of convergence to $f^*$, and the more recent \cite{slot-monique-2020-3} for rate of convergence of both upper and lower bounds on $\B=\{0,1\}^n$.

\section{Main result}

\subsection{Notation and definition}
Let $\R[\x]=\R[x_1,\ldots,x_n]$ be  the ring of real polynomials in the variables $x_1,\ldots,x_n$ and let $\R[\x]_t\subset\R[\x]$ be  its subspace of polynomials of degree at most $t$. Let $\N^n_t:=\{\alpha\in\N^n:\vert\alpha\vert\leq t\}$ where 
$\vert\alpha\vert=\sum_i\alpha_i$. For an arbitrary 
Borel subset $\mathcal{X}$ of $\R^n$, denote by $\mathscr{M}(\mathcal{X})_+$ the convex cone of finite Borel measures 
on $\mathcal{X}\subset\R^n$, and by $\mathscr{P}(\mathcal{X})$ is subset of probability measures on $\mathcal{X}$..

\subsection{Moment and localizing matrices}
Given an sequence $\y=(y_\alpha)_{\alpha\in\N^n}$ and polynomial
$g\in\R[\x]$,  $\x\mapsto g(\x):=\sum_\gamma g_\gamma\,\x^\gamma$, the \emph{localizing} matrix $\M_t(g\,\y)$ associated with $g$ and $\y$ is th real symmetric matrix with rows and columns indexed by 
$\alpha\in\N^n_t$ and with entries
\begin{equation}
\label{eq:localizing}
\M_t(g\,\y)(\alpha,\beta)\,:=\,\sum_\gamma g_\gamma\,y_{\alpha+\beta+\gamma}\,,\quad\alpha,\beta\in\N^n_t\,.
\end{equation}
If $g(\x)=1$ for all $\x$ then $\M_t(g\,\y)\,(=\M_t(\y))$ is called the \emph{moment} matrix.

A sequence $\y=(y_\alpha)_{\alpha\in\N^n}$ has a representing measure if there exists
a (positive) finite Borel measure $\phi$ on $\R^n$ such that
$y_\alpha=\int\x^\alpha\,d\phi$ for all $\alpha\in\N^n$.

If $\y$ has a representing measure supported on $\{\x:g(\x)\geq0\}$ then
$\M_t(\y)\succeq0$ and $\M_t(g\,\y)\succeq0$ for all $t\in\N$. The converse is not true in general; however, the following important result
is at the core of the Moment-SOS hierarchy.
\begin{thm}(Putinar \cite{putinar93})
\label{th-put}
Let $g_j\in\R[\x]$, $j=0,\ldots,m$ with $g_0(\x)=1$ for all $\x$, and
let $G:=\{\x\in\R^n: g_j(\x)\geq0,\,\:j=1,\ldots,m\,\}$ be compact. Moreover, assume that for some $M>0$, the quadratic polynomial $\x\mapsto M-\Vert\x\Vert^2$ can be written in the form $\sum_{j=0}^m\psi_j\,g_j$,
for some SOS polynomials $\psi_0,\ldots\psi_m$. 

Then a sequence
$\y=(y_\alpha)_{\alpha\in\N^n}$ has a representing measure on $G$
if and only if  $\M_t(g_j\,\y)\succeq0$ for all $t\in\N$, and all $j=0,\ldots,m$.
\end{thm}

\subsection*{Orthonormal polynomials}

Let $\B\subset\R^n$ be the compact basic semi-algebraic set defined in 
\eqref{set-B} assumed to have a nonempty interior. Let $\mu$ be a finite Borel (reference) measure whose support is exactly $\B$ and with
associated sequence of orthonormal polynomials $(T_\alpha)_{\alpha\in\N^n}\subset\R[\x]$. That is
\[\int_\B T_\alpha\,T_\beta\,d\mu\,=\,\delta_{\alpha=\beta}\,,\quad\forall\alpha,\beta\in\N^n\,.\]
For instance, if $\B=[-1,1]^n$ and $\mu$ is the uniform probability distribution
on $\B$, one may choose for the family $(T_\alpha)$ the tensorized \emph{Legendre} polynomials.
Namely if $(T_j)\subset\R[x]$ is the family of univariate Legendre polynomials, then
\[T_\alpha(\x)\,:=\,\prod_{j=1}^nT_{\alpha_j}(x_j)\,,\quad \alpha\in\N^n\,.\]
For every $t\in\N$, the mapping $K_t:\B\times\B\to\R$, 
\[(\x,\y)\mapsto K_t(\x,\y)\,:=\,\sum_{\vert\alpha\vert\leq t} T_\alpha(\x)\,T_\alpha(\y)\,,\quad\x,\y\in\B\,\]
is called the Cristoffel-Darboux kernel associated with $\mu$. An important 
property of $K_t$ is to \emph{reproduce} polynomials of degree at most $t$, that is:
\begin{equation}
\label{observe}
p(\x)\,=\,\int_\B p(\y)\,K_t(\x,\y)\,d\mu(\y)\,\quad\forall \x\in\B\,,\quad\forall p\in\R[\x]_t\,.
\end{equation}
This is why $K_t$ is called a reproducing kernel, and 
$\R[\x]_t$ viewed as a 
finite-dimensional vector subspace of the Hilbert space $L^2(\B,\mu)$, is called a Reproducing 
Kernel Hilbert Space (RKHS). For more details on the theory of orthogonal polynomials, the interested reader is referred to
e.g. \cite{dunkl} and the many references therein.

\subsection{Main result}
\subsection*{An observation}
Let $f\in\R[\x]$ and let $t\geq {\rm deg}(f)=d_f$ be fixed. Let
$\mathscr{P}(\B)\subset\mathscr{M}(\B)_+$ be the space of probability measures on $\B$. Then
\begin{eqnarray*}
f^*&=&
\displaystyle\min_{\phi\in\mathscr{P}(\B)_+}\,\int_\B f\,d\phi\\
&=&
\displaystyle\min_{\phi\in\mathscr{P}(\B)_+}\,\int_\B \int_\B f(\y)\,K_t(\x,\y)\,d\mu(\y)\,d\phi(\x)\\
&=&
\displaystyle\min_{\phi\in\mathscr{P}(\B)_+}\,\int_\B f(\y)\,(\int_\B K_t(\x,\y)\,d\phi(\x)\,)\,d\mu(\y)\\
&=&
\displaystyle\min_{\phi\in\mathscr{P}(\B)_+}\,\int_\B f(\y)\,
\left[\sum_{\vert\alpha\vert\leq t}(\underbrace{\int_\B T_\alpha(\x)\,d\phi(\x)}_{\sigma_\alpha})\,T_\alpha(\y)\,\right]\,
d\mu(\y)\\
&=&
\displaystyle\min_{\phi\in\mathscr{P}(\B)_+}\,\int_\B f(\y)\,
(\underbrace{\sum_{\vert\alpha\vert\leq t}\sigma_\alpha\,T_\alpha(\y)}_{\sigma_t(\y)\in\R[\x]_t})\,
d\mu(\y)\,,
\end{eqnarray*}
where the second equality follows from Fubini-Tonelli  interchange theorem valid in this simple setting.
In other words, we have proved the following:
\begin{lem}
Let $\B\subset\R^n$ be as in \eqref{set-B} and let $\mu$ be a finite Borel (reference) measure whose support is exactly $\B$ and with
associated sequence of orthonormal polynomials $(T_\alpha)_{\alpha\in\N^n}$.
Let $f^*=\min\,\{f(\x):\x\in\B\}$. Then for every fixed $t\geq {\rm deg}(f)$:
\begin{equation}
\label{eq:equiv}
f^*\,=\,\displaystyle\inf_{\sigma\in\R[\x]_t}\,
\int_\B f(\y)\,\sigma(\y)\,d\mu(\y)\,,
\end{equation}
where the infimum is over all polynomials $\sigma\in\R[\x]_t$ of the form:
\begin{eqnarray}
\label{eq:a1}
\sigma(\x)&=&\sum_{\vert\alpha\vert\leq t}\sigma_\alpha\,T_\alpha(\x)\,,\quad\forall\x\in\B\,\quad\mbox{with}\\
\label{eq:a2}
\sigma_\alpha&=&\int_\B T_\alpha(\x)\,d\phi(\x)\,,\quad\forall\alpha\in\N^n_t\,,\quad
\mbox{for some $\phi\in\mathscr{P}(\B)$.}
\end{eqnarray}
\end{lem}

So solving \eqref{eq:equiv} is equivalent to searching for a signed measure
$\sigma\,d\mu$ with polynomial (signed) density $\sigma\in\R[\x]_t$ that satisfies
\eqref{eq:a1}-\eqref{eq:a2}.

\subsection{A hierarchy of relaxations of \eqref{eq:equiv}}
In this section we show the SOS-hierarchy defined in \eqref{eq:lower}
is the dual semidefinite program of a natural SDP-relaxation of
\eqref{eq:equiv}. In fact the only difficult constraint in
\eqref{eq:equiv} is \eqref{eq:a2} which  demands $\bsig$ to admit a representing probability measure $\phi$ on $\B$.\\

Let $\D_t$ be the lower triangular matrix for the change of basis of $\R[\x]_{2t}$ from the monomial basis
$\v_{2t}(\x)=(\x^\alpha)_{\alpha\in\N^n_{2t}}$ of $\R[\x]_{2t}$ to the  basis $(T_\alpha)_\alpha$, i.e.,
\begin{equation}
\label{change-basis}
\left[\begin{array}{c}T_0\\ \cdots\\T_\alpha\\ \cdots\end{array}\right]\,=\,
\D_t\cdot
\left[\begin{array}{c}1\\ \cdots\\ \x^\alpha\\ \cdots\end{array}\right]\,=\,
\D_t\cdot\v_{2t}(\x)\,
\end{equation}
and denote $\D'_t$ the transpose of $\D_t$. The matrix $\D_t$ is nonsingular with positive diagonal.
Then with $\bsig=(\sigma_\alpha)_{\alpha\in\N^n_{2t}}$, \eqref{eq:a2} reads
\begin{equation}
\label{eq:a22}
\bsig\,=\,\D_t\cdot\y\quad\mbox{with}\quad\y=
\displaystyle\int_\B \v_{2t}(\x)\,d\phi(\x).
\end{equation}
That is, $\y=(y_\alpha)_{\alpha\in\N^n_{2t}}$ is required to be 
a moment sequence as it 
has a representing probability measure $\phi\in\mathscr{P}(\B)$.
So in view of Theorem \ref{th-put}, the constraint \eqref{eq:a22} can be relaxed to
\[\bsig=\D_t\cdot\y\quad\mbox{with $y_0=1$ and}
\quad\M_{t-d_j}(g_j\,\y)\,\succeq0\,,\quad j=0,\ldots,m\,.\]
Therefore, consider the following relaxation of \eqref{eq:equiv} 
\begin{equation}
\label{final-relax}
\begin{array}{rl}\rho_{2t}=\displaystyle\inf_{\sigma\in\R[\x]_{2t}}&\{\,\displaystyle\int_\B f(\x)\,(\displaystyle\sum_{\alpha\in\N^n_{2t}}\sigma_\alpha \,T_\alpha(\x))\,d\mu(\x)\::\:\bsig\,=\,\D_t\cdot\y\,;\\
&y_0=1\,;\:\M_{t-d_j}(g_j\,\y)\,\succeq0\,,\quad j=0,\ldots,m\,\}\,.
\end{array}
\end{equation}

\begin{lem}
\label{lem-final}
Let $\B\subset\R^n$ be as in \eqref{set-B} and let $\mu$ be a finite Borel (reference) measure whose support is exactly $\B$ and with
associated sequence of orthonormal polynomials $(T_\alpha)_{\alpha\in\N^n}$. The semidefinite relaxation \eqref{final-relax} of
\eqref{eq:equiv} reads:
\begin{equation}
\label{eq:recognize}
\inf_\y\,\{\,\langle \f,\y\rangle\::\:y_0=1\,;\:\M_{t-d_j}(g_j\,\y)\,\succeq\,0\,,\quad j=0,\ldots,m\,\}\,,
\end{equation}
which is the dual of \eqref{eq:lower}
\end{lem}
\begin{proof}
With $f(\x)=\sum_\alpha f_\alpha\,\x^\alpha =\langle\f,\v_{2t}(\x)\rangle$, write $f(\x)=\sum_{\alpha\in\N^n_{2t}}\tilde{f}_\alpha \,T_\alpha(\x)$ in the basis $(T_\alpha)_{\alpha\in\N^n_{2t}}$. Then
 with $\tilde{\mathbf{f}}=(\tilde{f}_\alpha)$ one obtains
\[\langle \tilde{\f},\D_t\cdot\v_{2t}(\x)\rangle \,=\,\langle\D'_t\,\tilde{\f},\v_{2t}(\x)\rangle
\,=\,\langle\f,\v_{2t}(\x)\rangle
\Rightarrow \tilde{\f}=(\D'_t)^{-1}\f\,.\]
Finally,  as the $T_\alpha$'s form an orthonormal basis, the criterion
\[\int_\B f(\x)\,(\displaystyle\sum_{\alpha\in\N^n_{2t}}\sigma_\alpha\,T_\alpha(\x))\,d\mu(\x)\,\]
to minimize in \eqref{final-relax} reads:
\[\int_\B f(\x)\,(\displaystyle\sum_{\alpha\in\N^n_{2t}}\sigma_\alpha\,T_\alpha(\x))\,d\mu(\x)\,=\,\langle \tilde{\f},\bsig\rangle\,=\,
\langle(\D'_t)^{-1}\f,\D_t\,\y\rangle\,=\,\langle \f,\y\rangle\,,\]
which yields that \eqref{final-relax} is exactly \eqref{eq:recognize}. Next, that \eqref{eq:recognize} is a dual of \eqref{eq:lower} is 
a standard result in polynomial optimization \cite{lass-siopt-01,lass-book}.  
\end{proof}
Of course by reverting the process of the above proof, the semidefinite program \eqref{eq:recognize}
can be transformed to \eqref{final-relax} once a reference measure $\mu$ with support exactly $\B$
is defined with its associated orthonormal polynomials $(T_\alpha)$. Indeed, once
$\mu$ and the $T_\alpha$'s are defined, one may use the change of basis matrix 
$\D$ in \eqref{change-basis} to pass from \eqref{eq:recognize} to \eqref{final-relax}.

\begin{cor}
Let $\B\subset\R^n$ be as in \eqref{set-B} and let $\mu$ be a finite Borel (reference) measure whose support is exactly $\B$ and with
associated sequence of orthonormal polynomials $(T_\alpha)_{\alpha\in\N^n}$. Let $f^*$
be the global minimum of $f$ on $\B$.

Let $t$ be such that  the semidefinite relaxation \eqref{final-relax} (or equivalently \eqref{eq:recognize}) is exact, i.e., if $\rho_{2t}=f^*$.
If an optimal solution $\y^*$ of \eqref{eq:recognize} has a representing measure $\phi^*\in\mathscr{M}(\B)_+$, then 
an optimal polynomial density $\sigma^*\in\R[\x]_{2t}$ of \eqref{final-relax} satisfies:
\[\sigma^*(\bxi)\,=\,\sum_{\alpha\in\N^n_{2t}} T_\alpha(\bxi)^2\,=\,K_{2t}(\bxi,\bxi)\,,\quad\forall \bxi\in{\rm supp}(\phi^*)\,,\]
that is, $\sigma^*(\bxi)^{-1}$ is the Christoffel function evaluated at the global minimizer $\bxi\in\B$.
\end{cor}
\begin{proof}
If $\y^*$ has a representing measure $\phi^*\in\mathscr{M}(\B)_+$ then necessarily
$f(\bxi)=f^*$ for all $\bxi\in{\rm sup}(\phi^*)$; see e.g. \cite{lass-siopt-01,lass-book}. In particular,
for every $\bxi\in{\rm sup}(\phi^*)$, the vector $\hat{\y}:=(\bxi^\alpha)_{\alpha\in\N^n_{2t}}$ is also an optimal solution of \eqref{eq:recognize}. Then
\[\bsig^*\,=\,\D_t\cdot\hat{\y}=\D_t\cdot\v_{2t}(\bxi)\,=\,
\left[\begin{array}{c}T_0(\bxi)\\ \cdots\\T_\alpha(\bxi)\\ \cdots\end{array}\right]\,,\]
i.e., $\sigma^*_\alpha=T_\alpha(\bxi)$ for all $\alpha\in\N^n_{2t}$. Therefore,
\[\x\mapsto \sigma^*(\x)\,=\,\sum_{\alpha\in\N^n_{2t}}T_\alpha(\bxi)\,T_\alpha(\x)\,=\,\K_{2t}(\bxi,\x)\,,\]
and so $\sigma^*(\bxi)=\K_{2t}(\bxi,\bxi)$. In other words, $\sigma^*(\bxi)^{-1}$ is the 
Christoffel function associated with $\mu$, evaluated at $\bxi\in\B$.
\end{proof}
\subsection*{Discussion} Observe that the formulation \eqref{final-relax} does not require
that the set $\B$ is a ``simple" set as it is required in \eqref{eq:upper}. Indeed the 
orthonormal polynomials $(T_\alpha)$ are only used to provide an \emph{interpretation} of the hierarchy of lower bounds \eqref{eq:recognize} (and its dual \eqref{eq:lower}).
On the other hand, for the hierarchy of upper bounds \eqref{eq:upper}, $\B$ indeed needs to be a ``simple" set for 
computational purposes. This is because one needs the numerical value of the moments of $\mu$ 
for a practical implementation of  \eqref{eq:upper}.

Lemma \ref{lem-final} shows that the Moment-SOS hierarchy described in \cite{lass-siopt-01,lass-book} amounts to compute
a hierarchy of signed polynomial densities with respect to some reference measure $\mu$ with support exactly $\B$. When the step-$t$ relaxation is exact (which takes place generically \cite{nie}) the resulting optimal density $\sigma$ in \eqref{final-relax}
is nothing less than the polynomial $\x\mapsto K_t(\bxi,\x)$ where $\bxi$ is a global minimizer of $f$ on $\B$, $K_t(\bxi,\x)$ is
the celebrated Cristoffel-Darboux kernel in approximation theory, and $\sigma(\bxi,\bxi)$ is the reciprocal of the Christoffel function 
evaluated at a global minimizer $\bxi$.

\section{Conclusion}

We have shown that the Moment-SOS hierarchy that provides an increasing sequence of lower bounds on the global minimum
of a polynomial $f$ on a compact set $\B$, has a simple interpretation related to orthogonal polynomials
associated with an arbitrary reference measure whose support is exactly $\B$.
This interpretation strongly relates  polynomial optimization (here the Moment-SOS hierarchy) with
the Christoffel-Darboux kernel and  the Christoffel function, 
fundamental tools in the theory of orthogonal polynomials and the theory of of approximation. 

 It is another item in the list of previous contributions  \cite{lass-2020,etienne-monique-2020,slot-monique-2020,slot-monique-2020-2} 
 that also link some issues in polynomial optimization with orthogonal polynomials associated with appropriate measures. We hope that such connections 
 will stimulate even further investigations in this direction.

\end{document}